\documentclass[11pt]{amsart}
\usepackage{amssymb,amsmath}
\usepackage{tikz}

\newtheorem{theorem}{Theorem}[section]
\newtheorem{proposition}[theorem]{Proposition}
\newtheorem{lemma}[theorem]{Lemma}

\newtheorem{example}[theorem]{Example}

\numberwithin{equation}{section}

\newcommand{\Nnn}{{\mathbb N}}

\DeclareMathOperator{\mult}{mult}
\DeclareMathOperator{\wt}{wt}

\begin{document}

\vspace*{-20 mm}

\title{Cyclotomic enumeration of polynomials}
\author{Richard EHRENBORG}

\address{Department of Mathematics, University of Kentucky, Lexington,
  KY 40506-0027.\hfill\break \tt http://www.math.uky.edu/\~{}jrge/,
  richard.ehrenborg@uky.edu.}



\begin{abstract}
Using the cyclotomic identity we compute sums
over $d$-tuples of monic polynomials in $F_{q}[x]$
weighted by the multiplicity of their irreducible factors.
As consequences we determine explicit expressions for
the number of $d$-tuples of polynomials such that their
greatest common divisor is $r$th power free.
We also compute the number of monic polynomials
where the multiplicity of each irreducible factor
belongs to the monoid generated by two relatively prime integers.
\end{abstract}

\maketitle

\section{Introduction}

Let $F_{q}$ be the finite field of cardinality $q$ where $q$ is a prime power.
The number of monic square-free polynomials in the polynomial ring
$F_{q}[x]$ of degree $n$
is given by $q^{n} - q^{n-1}$ for $n \geq 2$.
For this result and its connection
to algebraic geometry; see the papers~\cite{Church_Ellenberg_Farb,Fulman}
and references therein.
Alegre, Juarez and Pajela
extended this enumeration to count
the number of monic polynomials of degree $n \geq r$
that are $r$th power free
and their result is the equally beautiful
$q^{n} - q^{n+1-r}$; see~\cite{Alegre_Juarez_Pajela}.
In this paper we further extend these results to
enumerate other classes of polynomials
defined in terms of multiplicities of their irreducible factors.
Two noteworthy results are as follows:
First, the number of $d$-tuples of monic polynomials
such that their greatest common divisor is $r$th power free
is given by $q^{N} - q^{N+1-dr}$,
where $N$ is the sum of the degrees and
each individual degree is at least~$r$;
see Proposition~\ref{proposition_r_d}.
Second, the number of monic polynomials
where no irreducible factor has multiplicity $1$ is given by
$q^{\lfloor n/2 \rfloor}
+
q^{\lfloor n/2 \rfloor - 1}
-
q^{\lfloor (n-1)/3 \rfloor}$
where the degree $n$ is at least $2$;
see Example~\ref{example_no_multiplicity_1}.
This last example is a special case where the allowed multiplicities
belong to a monoid generated by two relatively prime integers.

The paper is organized as follows.
In Section~\ref{section_Preliminaries}
we review the cyclotomic identity which is the engine behind all
these results. 
In Section~\ref{section_Weighted}
we show that any multivariate generating function with constant coefficient $1$
factors as an infinite product of the form
$\prod_{{\bf i} > {\bf 0}} 1/\left(1 - {\bf z}^{{\bf i}}\right)^{a_{\bf i}}$;
see equation~\eqref{equation_product}.
We use this form of the generating function to obtain the main result
on weighted enumeration of $d$-tuples of monic polynomials.
In the case of a single polynomial, that is $d=1$, this result was also
obtained by Stanley~\cite{Stanley_Some}.
In Section~\ref{section_Enumerative_results}
we apply Theorem~\ref{theorem_main} to enumerate $d$-tuples of monic polynomials
where the multiplicity vector belongs to certain subsets of~$\Nnn^{d}$.
In Section~\ref{section_Two_results} we offer two results on the number
of irreducible factors.
We end the paper with open problems.

\section{Preliminaries}
\label{section_Preliminaries}

We begin by stating the well-known enumeration of
the number of irreducible polynomials with coefficients in the finite field $F_{q}$.
To make this note self-contained we include a brief proof.
\begin{lemma}
The number of monic irreducible polynomials in $F_{q}[x]$
of degree $n$ is given by
\begin{align*}
M(q,n)
& = 
\frac{1}{n} \cdot \sum_{k | n} \mu(n/k) \cdot q^{k},
\end{align*}
where $\mu$ denotes the classical M\"obius function.
\end{lemma}
\begin{proof}
Consider the field $F_{q^{n}}$ as a vector space over the field $F_{q}$
and consider the subspace arrangement consisting of all the subfields of $F_{q^{n}}$,
that is, $\{ F_{q^{k}} \}_{k | n, k < n}$.
Observe that the intersection lattice is the divisor lattice of the positive integer~$n$.
By the finite field method of Athanasiadis~\cite{Athanasiadis}
or the principle of inclusion--exclusion~\cite{EC1},
we know that the number of points
not in any subfield
in~$F_{q^{n}}$ is given by 
\begin{align*}
\left|
F_{q^{n}}
-
\bigcup_{k | n, k < n} F_{q^{k}}
\right|
& =
\sum_{k | n} \mu(n/k) \cdot |F_{q^{k}}|
=
\sum_{k | n} \mu(n/k) \cdot q^{k} ,
\end{align*}
where $\mu$ denotes the classical, that is,
the number--theoretic, M\"obius function.
Finally for such a point~$a$ we know that
the $n$ elements
$a, a^{q}, a^{q^{2}}, \ldots, a^{q^{n-1}}$
form the roots of an irreducible polynomial in $F_{q}[x]$
of degree $n$, explaining the factor $1/n$.
\end{proof}

We next present the classical cyclotomic identity.
There are many proofs of this identity,
such as those using free Lie algebras~\cite[Eq.\ (5.3.17)]{Lothaire},
Lyndon words~\cite[Theorem~5.1.5]{Lothaire}
or bijections~\cite{Metropolis_Rota_1,Metropolis_Rota_2}.
However, we are mainly interested in
the monic polynomials in the ring~$F_{q}[x]$.
Hence we present a brief proof of the cyclotomic identity
using this polynomial ring and
its uniqueness of factorization.
This proof can be found in~\cite{Reutenauer}
and~\cite[Lemma~2.1]{Fulman}.
\begin{theorem}[The Cyclotomic Identity]
The following identity for formal power series holds
\begin{align*}
\frac{1}{1-qz}
& =
\prod_{k \geq 1}
\left(
\frac{1}{1 - z^{k}}
\right)^{M(q,k)} .
\end{align*}
\end{theorem}
\begin{proof}
The left-hand side generating function enumerates the
number of monic polynomials of degree~$n$ in~$F_{q}[n]$.
Note that
$1/(1 - z^{k})$ is the generating function for enumerating the powers
of one irreducible polynomial of degree $k$.
Hence
$1/(1 - z^{k})^{M(q,k)}$ is the generating function for all possible products
of the $M(q,k)$ irreducible polynomials of degree $k$.
Hence the right-hand side is the generating function for all
products of irreducible polynomials.
The identity follows by the uniqueness of factorization in~$F_{q}[x]$.
\end{proof}

\section{Weighted enumerating according to multiplicities}
\label{section_Weighted}

For a vector ${\bf n} = (n_{1}, \ldots, n_{d})$
in $\Nnn^{d}$
define the monomial
${\bf z}^{{\bf n}} = z_{1}^{n_{1}} \cdots z_{d}^{n_{d}}$.
Let $w_{{\bf n}} = w_{(n_{1}, \ldots, n_{d})}$
be weights and form the generating function
\begin{align*}
g({\bf z})
& =
\sum_{{\bf n} \geq {\bf 0}} w_{{\bf n}} \cdot {\bf z}^{{\bf n}} .
\end{align*}
We furthermore assume that
$w_{{\bf 0}} = w_{(0, \ldots, 0)} = 1$, that is,
the constant term of $g({\bf z})$ is $1$.
Next we show that the generating function $g({\bf z})$
can be expressed as an infinite product.
\begin{lemma}
The generating function $g({\bf z})$ with constant term $1$
can be uniquely written as
\begin{align}
g({\bf z})
& =
\prod_{{\bf i} > {\bf 0}} \left(\frac{1}{1 - {\bf z}^{{\bf i}}}\right)^{a_{\bf i}} .
\label{equation_product}
\end{align}
\end{lemma}
\begin{proof}
To observe this fact, take the logarithm on both sides
\begin{align*}
\sum_{{\bf n} \geq {\bf 0}} b_{{\bf n}} \cdot {\bf z}^{{\bf n}} 
=
\log(g({\bf z}))
& =
\sum_{{\bf i} > {\bf 0}} a_{\bf i} \cdot \log\left(\frac{1}{1 - {\bf z}^{{\bf i}}}\right) 
=
\sum_{{\bf i} > {\bf 0}} a_{\bf i} \cdot \sum_{k \geq 1}  \frac{1}{k} \cdot {\bf z}^{k \cdot {\bf i}} .
\end{align*}
Note that $w_{{\bf 0}} = 1$ implies that $b_{{\bf 0}} = 0$.
By comparing the coefficients of ${\bf z}^{{\bf n}}$ and
multiplying with $\gcd({\bf n})$
we have
\begin{align*}
\gcd({\bf n}) \cdot b_{{\bf n}}
& =
\sum_{k | \gcd({\bf n})}
\frac{\gcd({\bf n})}{k} \cdot a_{{\bf n}/k} 
=
\sum_{k | \gcd({\bf n})}
\gcd({\bf n}/k) \cdot a_{{\bf n}/k} .
\end{align*}
Hence by M\"obius inversion and dividing by $\gcd({\bf n})$ we have
\begin{align*}
a_{{\bf n}}
& =
\frac{1}{\gcd({\bf n})}
\cdot
\sum_{k | \gcd({\bf n})}
\mu(k) \cdot \gcd({\bf n}/k) \cdot b_{{\bf n}/k} 
=
\sum_{k | \gcd({\bf n})}
\frac{\mu(k)}{d} \cdot b_{{\bf n}/k} ,
\end{align*}
showing that the factorization in~\eqref{equation_product} exists
and is unique.
\end{proof}

Let $F_{q}^{\text{monic}}[x]$
denote the set of all
{\em monic} polynomials in $F_{q}[x]$.
For two polynomials $p$ and $s$ in $F_{q}^{\text{monic}}[x]$,
where $s$ is irreducible, let
$\mult(p;s)$ denote the multiplicity of $s$ in the polynomial $p$.
That is, the polynomial $p$ factors as
\begin{align*}
p
& =
\prod_{s} s^{\mult(p;s)} ,
\end{align*}
where the product is over all monic irreducible polynomials $s$.
Similarly for a list ${\bf p} = (p_{1}, \ldots, p_{d}) \in F_{q}^{\text{monic}}[x]^{d}$
let
$\mult({\bf p};s)
= 
(\mult(p_{1};s), \ldots, \mult(p_{d};s))$
and
$\deg({\bf p})
= 
(\deg(p_{1}), \ldots, \deg(p_{d}))$.
For a list ${\bf p}$ in $F_{q}^{\text{monic}}[x]^{d}$
define its {\em weight} (relative to $w$ as above) to be the product
\begin{align*}
\wt({\bf p})
& =
\prod_{s} w_{{\mult({\bf p};s)}} ,
\end{align*}
where $s$ ranges over all irreducible polynomials $s$.
Note that the product is well-defined since we assumed
that $w_{{\bf 0}}$ is $1$
and hence only a finite number of factors are not equal to~$1$.

For a vector ${\bf n} \in \Nnn^{d}$
let $\alpha_{q}({\bf n})$ denote the sum of the weights
of $d$-tuples of monic polynomials in $F_{q}^{\text{monic}}[x]^{d}$
of degree ${\bf n}$, that is, 
\begin{align*}
\alpha_{q}({\bf n})
& =
\sum_{\substack{{\bf p} \in F_{q}^{\text{monic}}[x]^{d} \\ \deg({\bf p}) = {\bf n}}}
\wt({\bf p}) .
\end{align*}
Note that if the weights $w_{{\bf n}}$ are all
$0$ or $1$ then
$\alpha_{q}({\bf n})$ enumerates $d$-tuples of monic polynomials
such that the multiplicity vector
$\mult({\bf p};s)$ belongs to the set
$\{{\bf n} \in \Nnn^{d} : w_{{\bf n}} = 1\}$
for all irreducible polynomials~$s$.

\begin{theorem}
The generating function for the sum of the weights
of the $d$-tuples of monic polynomials
is given by
\begin{align*}
\sum_{{\bf n} \geq {\bf 0}}
\alpha_{q}({\bf n}) \cdot {\bf z}^{{\bf n}}
& =
\prod_{{\bf i} > {\bf 0}} \left(\frac{1}{1 - q \cdot {\bf z}^{{\bf i}}}\right)^{a_{\bf i}} ,
\end{align*}
where the powers $a_{{\bf i}}$ are given by equation~\eqref{equation_product}.
\label{theorem_main}
\end{theorem}
\begin{proof}
After the simultaneous substitutions $z_{i} \longmapsto z_{i}^{k}$ we have
\begin{align*}
\sum_{{\bf n} \geq {\bf 0}} w_{{\bf n}} \cdot {\bf z}^{k \cdot {\bf n}}
& = 
\prod_{{\bf i} > {\bf 0}} \left(\frac{1}{1 - {\bf z}^{k \cdot {\bf i}}}\right)^{a_{\bf i}} .
\end{align*}
This quantity corresponds to one irreducible polynomial of degree $k$.
Multiplying over all irreducible polynomials~$s$
we obtain the sought after generating function
\begin{align*}
\sum_{{\bf n} \geq {\bf 0}}
\alpha_{q}({\bf n}) \cdot {\bf z}^{{\bf n}}
& =
\prod_{s}
\sum_{{\bf n} \geq {\bf 0}} w_{{\bf n}} \cdot {\bf z}^{\deg(s) \cdot {\bf n}} \\
& =
\prod_{k \geq 1}
\left(
\sum_{{\bf n} \geq {\bf 0}} w_{{\bf n}} \cdot {\bf z}^{k \cdot {\bf n}}
\right)^{M(q,k)} \\
& =
\prod_{k \geq 1}
\left(
\prod_{{\bf i} > {\bf 0}} \left(\frac{1}{1 - {\bf z}^{k \cdot {\bf i}}}\right)^{a_{\bf i}}
\right)^{M(q,k)} \\
& =
\prod_{{\bf i} > {\bf 0}}
\left(
\prod_{k \geq 1}
\left(\frac{1}{1 - {\bf z}^{k \cdot {\bf i}}}\right)^{M(q,k)}
\right)^{a_{\bf i}} .
\end{align*}
Now applying the cyclotomic identity with $z = {\bf z}^{{\bf i}}$
for each ${\bf i} > {\bf 0}$ yields the desired result.
\end{proof}

See Stanley~\cite[Section~6]{Stanley_Some}
for the case $d=1$ of this result.

\section{Enumerative results}
\label{section_Enumerative_results}

We now give explicit examples of applications of Theorem~\ref{theorem_main}.

\begin{proposition}
Let $d$ and $r$ be two positive integers.
The number of $d$-tuples ${\bf p} = (p_{1}, \ldots, p_{d})$
of monic polynomials such that
their degrees are given by the vector ${\bf n} = (n_{1}, \ldots, n_{d})$
and
their greatest common divisor
$\gcd(p_{1}, \ldots, p_{d})$
is $r$th power free
is given by
\begin{align}
\alpha_{q}({\bf n})
& =
\begin{cases}
q^{N} - q^{N +1 - rd} 
&
\text{ if }
\min(n_{1}, \ldots, n_{d}) \geq r , \\
q^{N}
&
\text{ if }
\min(n_{1}, \ldots, n_{d}) \leq r-1 , 
\end{cases}
\label{equation_M_1}
\end{align}
where $N$ denotes the sum $n_{1} + \cdots + n_{d}$.
\label{proposition_r_d}
\end{proposition}
\begin{proof}
We use the weights $w_{{\bf n}}$,
given by $1$ if $\min(n_{1}, \ldots, n_{d}) < r$ and $0$ otherwise.
Next note that
\begin{align}
\nonumber
\sum_{n_{1}, \ldots, n_{d} \geq 0}
z_{1}^{n_{1}} \cdots z_{d}^{n_{d}}
& =
\frac{1}{1-z_{1}} \cdots \frac{1}{1-z_{d}} , \\
\label{equation_min<r}
g({\bf z})
=
\sum_{\substack{n_{1}, \ldots, n_{d} \geq 0 \\ \min(n_{1}, \ldots, n_{d}) < r}}
z_{1}^{n_{1}} \cdots z_{d}^{n_{d}}
& =
\frac{1 - z_{1}^{r} \cdots z_{d}^{r}}{(1-z_{1}) \cdots (1-z_{d})} .
\end{align}
We apply Theorem~\ref{theorem_main}
to the generating function $g({\bf z})$:
\begin{align*}
\frac{1 - q \cdot z_{1}^{r} \cdots z_{d}^{r}}{(1 - q \cdot z_{1}) \cdots (1 - q \cdot z_{d})} 
& =
(1 - q \cdot z_{1}^{r} \cdots z_{d}^{r})
\cdot
\sum_{n_{1}, \ldots, n_{d} \geq 0}
q^{n_{1} + \cdots + n_{d}} \cdot
z_{1}^{n_{1}} \cdots z_{d}^{n_{d}} \\
& =
\sum_{\substack{n_{1}, \ldots, n_{d} \geq 0 \\ \min(n_{1}, \ldots, n_{d}) < r}}
q^{n_{1} + \cdots + n_{d}} \cdot
z_{1}^{n_{1}} \cdots z_{d}^{n_{d}} \\
& +
\sum_{\substack{n_{1}, \ldots, n_{d} \geq 0 \\ \min(n_{1}, \ldots, n_{d}) \geq r}}
(1 - q^{1-rd}) \cdot
q^{n_{1} + \cdots + n_{d}} \cdot
z_{1}^{n_{1}} \cdots z_{d}^{n_{d}} .
\end{align*}
The result follows by reading off the coefficient of ${\bf z}^{\bf n}$.
\end{proof}

\begin{example}
{\rm
Setting $r=1$ in
Proposition~\ref{proposition_r_d}
we obtain that
the number of $d$-tuples of monic polynomials 
of degree $(n_{1}, \ldots, n_{d})$
that are relatively prime
is given by
\begin{align}
\begin{cases}
q^{N} - q^{N +1 - d} 
&
\text{ if }
\min(n_{1}, \ldots, n_{d}) \geq 1 , \\
q^{N}
&
\text{ if }
\min(n_{1}, \ldots, n_{d}) = 0. 
\end{cases}
\label{equation_M_r=1}
\end{align}
}
\label{example_M_r=1}
\end{example}

\begin{example}
{\rm
Similarly, setting $d=1$ in the proposition
we obtain the Alegre--Juarez--Pajela result~\cite{Alegre_Juarez_Pajela},
that the number of $r$th power free monic polynomials of degree $n$ is given by
\begin{align}
\begin{cases}
q^{n} - q^{n +1 - r} 
&
\text{ if }
n \geq r , \\
q^{n}
&
\text{ if }
n \leq r-1 . 
\end{cases}
\label{equation_M_d=1}
\end{align}
}
\label{example_d=1}
\end{example}

\begin{proposition}
Let $d$, $r$ and $m$ be positive integers such that $r < m$.
Then the number of $d$-tuples of monic polynomials $(p_{1}, \ldots, p_{d})$
of degree ${\bf n} = (n_{1}, \ldots, n_{d})$
such that the multiplicity of each irreducible factor
in $\gcd(p_{1}, \ldots, p_{d})$ is
congruent to one of $0,1, \ldots, r-1 \bmod m$
is given by
\begin{align}
\alpha_{q}({\bf n})
& =
\begin{cases}
q^{N}
\cdot
\left(1 + q^{1-md} + \cdots + q^{\lfloor \min({\bf n})/m \rfloor \cdot (1-md)}\right)
& \text{ if } \min({\bf n}) < r, \\[5 mm]
q^{N}
\cdot
\left(1 + q^{1-md} + \cdots + q^{\lfloor \min({\bf n})/m \rfloor \cdot (1-md)}\right)
& \\
-
q^{N - rd + 1}
\cdot
\left(1 + q^{1-md} + \cdots + q^{\lfloor (\min({\bf n})-r)/m \rfloor \cdot (1-md)}\right)
& \text{ if } \min({\bf n}) \geq r,
\end{cases}
\label{equation_d_r_m}
\end{align}
where $N$ is again the sum $n_{1} + \cdots + n_{d}$.
\end{proposition}
\begin{proof}
The associated generating function $g({\bf z})$ is given by
\begin{align*}
\sum_{\substack{n_{1}, \ldots, n_{d} \geq 0 \\ \min(n_{1}, \ldots, n_{d}) \equiv 0,1, \ldots r-1 \bmod m}}
z_{1}^{n_{1}} \cdots z_{d}^{n_{d}}
& =
\frac{1}{1 - z_{1}^{m} \cdots z_{d}^{m}}
\cdot
\sum_{\substack{n_{1}, \ldots, n_{d} \geq 0 \\ \min(n_{1}, \ldots, n_{d}) < r}}
z_{1}^{n_{1}} \cdots z_{d}^{n_{d}} \\
& =
\frac{1 - z_{1}^{r} \cdots z_{d}^{r}}{(1 - z_{1}^{m} \cdots z_{d}^{m}) \cdot (1-z_{1}) \cdots (1-z_{d})} ,
\end{align*}
where the last step is by equation~\eqref{equation_min<r}.
Note that
\begin{align*}
\sum_{{\bf n} \geq {\bf 0}} c_{{\bf n}} \cdot {\bf z}^{{\bf n}}
& =
\frac{1}{(1 - q \cdot z_{1}^{m} \cdots z_{d}^{m}) \cdot (1 - q \cdot z_{1}) \cdots (1 - q \cdot z_{d})} \\
& =
\left(
\sum_{k \geq 0} q^{k} \cdot (z_{1} \cdots z_{d})^{mk}
\right)
\cdot
\left(
\sum_{{\bf n} \geq {\bf 0}} q^{N} \cdot {\bf z}^{\bf n} 
\right) \\
& =
\sum_{{\bf n} \geq {\bf 0}}
q^{N}
\cdot
\left(1 + q^{1-md} + \cdots + q^{\lfloor \min({\bf n})/m \rfloor \cdot (1-md)}\right)
\cdot
{\bf z}^{\bf n} .
\end{align*}
Hence the coefficient of ${\bf z}^{\bf n}$ in the generating function
\begin{align*}
\sum_{{\bf n} \geq {\bf 0}} \alpha_{q}({\bf n}) \cdot {\bf z}^{{\bf n}}
& =
\frac{1 - q \cdot z_{1}^{r} \cdots z_{d}^{r}}
{(1 - q \cdot z_{1}^{m} \cdots z_{d}^{m}) \cdot (1 - q \cdot z_{1}) \cdots (1 - q \cdot z_{d})} 
\end{align*}
is given by
\begin{align*}
\alpha_{q}({\bf n})
& =
\begin{cases}
c_{{\bf n}}
& \text{ if } \min({\bf n}) < r , \\
c_{{\bf n}} - q \cdot c_{{\bf n} - (r, \ldots, r)}
& \text{ if } \min({\bf n}) \geq r , \\
\end{cases}
\end{align*}
which is the expression in equation~\eqref{equation_d_r_m}.
\end{proof}

\begin{proposition}
Let $a$ and $b$ be two relatively prime positive integers.
Let $M$ be the monoid $\{a \cdot i + b \cdot j : i,j \in \Nnn\}$.
Then the number of monic polynomials in $F_{q}[x]$ of degree~$n$ such
that the multiplicity of each irreducible factor belongs to
the monoid~$M$ is given by
\begin{align}
\alpha_{q}(n)
& =
\sum_{\substack{i, j \geq 0 \\ a \cdot i + b \cdot j = n}} q^{i+j}
-
\sum_{\substack{i, j \geq 0 \\ a \cdot i + b \cdot j = n-ab}} q^{i+j+1} .
\label{equation_a_b}
\end{align}
\label{proposition_a_b}
\end{proposition}
\begin{proof}
The associated generating function $g(z)$ is given by
\begin{align*}
g(z)
& =
\sum_{n \in M} z^{n}
=
\frac{1- z^{ab}}{(1- z^{a}) \cdot (1- z^{b})} ;
\end{align*}
see for instance~\cite[Exercise~1.33]{Beck_Robins}.
To apply Theorem~\ref{theorem_main}
note that
\begin{align*}
\frac{1}{(1- q \cdot z^{a}) \cdot (1- q \cdot z^{b})}
& =
\sum_{n \geq 0}
\sum_{\substack{i, j \geq 0 \\ a \cdot i + b \cdot j = n}} q^{i+j} \cdot z^{n} .
\end{align*}
Multiplying with $1- q \cdot z^{ab}$ yields
the desired generating function
\begin{align}
\sum_{n \geq 0} \alpha_{q}(n) \cdot z^{n}
=
\frac{1- q \cdot z^{ab}}{(1- q \cdot z^{a}) \cdot (1- q \cdot z^{b})}
& =
(1- q \cdot z^{ab})
\cdot
\bigg(
\sum_{n \geq 0}
\sum_{\substack{i, j \geq 0 \\ a \cdot i + b \cdot j = n}} q^{i+j} \cdot z^{n}
\bigg)
\label{equation_generating_function_a_b}
\end{align}
and hence the coefficient of $z^{n}$ is given by
equation~\eqref{equation_a_b}.
\end{proof}

Assume that $n$ is large enough and that
we have a pair $(i,j)$ of integers such that
$a \cdot i + b \cdot j = n$.
Then the index of every summand in the first sum in~\eqref{equation_a_b}
is of the form
$(i + k \cdot b, j - k \cdot a)$
and thus the powers of $q$ in this sum
are of the form
$i + j + k \cdot (b-a)$.
That is, the powers lie in the infinite arithmetic progression
containing $i+j$ with step size $b-a$.
Similarly, the summands in the second sum in~\eqref{equation_a_b}
are indexed by pairs of the form
$(i - b + k \cdot b, j - k \cdot a)$,
and hence the powers in this sum
are of the form
$i + j + 1 - b + k \cdot (b-a)$.
Hence if any cancellation occurs between the two sums,
we need $b-1$ to be divisible by $b-a$, that is,
$b \equiv 1 \bmod b-a$.
Note that this condition is equivalent to
$a \equiv 1 \bmod b-a$.
In this case we present an improved statement without
any cancellation. In order to do so, we follow the notation
of~\cite[Chapter~1]{Beck_Robins}:
Let $a^{-1}$ and $b^{-1}$ denote two integers such that
\begin{align*}
a \cdot a^{-1} & = 1 \bmod b,  &
b \cdot b^{-1} & = 1 \bmod a .
\end{align*}
Furthermore, let $\{x\}$ denote the fractional part of the real number $x$,
that is, $\{x\} = x - \lfloor x \rfloor$.

\begin{proposition}
Let $a$ and $b$ be two relatively prime positive integers
such that $a < b$ and $b \equiv a \equiv 1 \bmod b-a$.
Let $M$ be the monoid $\{a \cdot i + b \cdot j : i,j \in \Nnn\}$
and let $n > a \cdot b - a - b$.
Then the number of monic polynomials in $F_{q}[x]$ of degree~$n$ such
that the multiplicity of each irreducible factor belongs to
the monoid~$M$ is given by
\begin{align}
\alpha_{q}(n)
& =
\frac{q^{b} - q}{q^{b} - q^{a}}
\cdot
q^{n/a - (b-a) \cdot \{b^{-1} \cdot n / a\}}
-
\frac{q^{a} - q}{q^{b}-q^{a}} 
\cdot
q^{n/b + (b-a) \cdot \{a^{-1} \cdot n / b\}} .
\label{equation_a_b_improved}
\end{align}
\label{proposition_a_b_improved}
\end{proposition}
\begin{proof}
Assume that $n$ is large.
Let $(i_{0},j_{0})$ be given by
$i_{0} = n/a - b \cdot \{b^{-1} \cdot n / a\}$
and
$j_{0} = a \cdot \{b^{-1} \cdot n / a\}$.
Then the pair $(i_{0}, j_{0})$ satisfies
$a \cdot i + b \cdot j = n$,
and $j_{0}$ is the smallest such nonnegative integer $j$
and $i_{0}$ is the largest such integer~$i$.
Since $a < b$ we know that $i_{0}+j_{0}$ is the maximum over all
$i+j$ where $a \cdot i + b \cdot j = n$.
Similarly, the pair $(i_{0}-b,j_{0})$
is the corresponding such pair for the equation
$a \cdot i + b \cdot j = n - ab$.
Hence the powers in the first sum in equation~\eqref{equation_a_b}
are a finite arithmetic progression starting at $i_{0}+j_{0}$ and decreasing with step $b-a$.
Similarly, the powers in the second sum in equation~\eqref{equation_a_b}
are an arithmetic progression starting at $i_{0}+j_{0}-b$ and also decreasing with step $b-a$.
Note that 
$i_{0} + j_{0} = n/a - (b-a) \cdot \{b^{-1} \cdot n / a\}$.
Thus the positive terms in the difference in equation~\eqref{equation_a_b}
are given by
\begin{align}
\label{equation_up}
q^{i_{0}+j_{0}}
+
q^{i_{0}+j_{0} - (b-a)}
+
\cdots
+
q^{i_{0}+j_{0} + 1 - a}
& =
\frac{q^{1-b} - 1}{q^{a-b} - 1}
\cdot
q^{i_{0}+j_{0}} \\
\nonumber
& =
\frac{q^{b} - q}{q^{b} - q^{a}}
\cdot
q^{n/a - (b-a) \cdot \{b^{-1} \cdot n / a\}} .
\end{align}
Similarly, the negative terms in equation~\eqref{equation_a_b}
is obtained by switching the roles of $a$ and $b$
and also switching $n$ and $n-ab$.
That is, let
$i_{1} = b \cdot \{a^{-1} \cdot (n-ab) / b\} = b \cdot \{a^{-1} \cdot n / b\}$
and
$j_{1} = (n-ab)/b - a \cdot \{a^{-1} \cdot (n-ab) / b\}
= n/b - a - a \cdot \{a^{-1} \cdot n / b\}$.
Then $a \cdot i_{1} + b \cdot j_{1} = n - a b$,
and $q^{i_{1}+j_{1}+1}$ is the smallest power occurring in the second sum in~\eqref{equation_a_b}.
Similarly, 
$a \cdot i_{1} + b \cdot (j_{1}+a) = n$,
and
$q^{i_{i} + (j_{1}+a)}$ is the smallest power in the first sum in~\eqref{equation_a_b}.
After cancelling, the remaining negative terms are
\begin{align}
\label{equation_down}
&
q^{i_{1}+j_{1}+1}
+
q^{i_{1}+j_{1}+1+(b-a)}
+
\cdots
+
q^{i_{1}+j_{1}+a-(b-a)} \\
\nonumber
& =
(q^{1-a} + q^{1-a+(b-a)} + \cdots + q^{a-b})
\cdot 
q^{n/b + (b-a) \cdot \{a^{-1} \cdot n / b\}} \\
\nonumber
& =
\frac{q^{a}-q}{q^{b}-q^{a}}
\cdot 
q^{n/b + (b-a) \cdot \{a^{-1} \cdot n / b\}} .
\end{align}
Finally, we have to motivate the lower bound on $n$.
Note that the expression in equation~\eqref{equation_up}
increases by a multiplicative factor of $q$ when $n$ increases by $a$.
Similarly, 
the expression in equation~\eqref{equation_down}
increases by a factor of $q$ when $n$ increases by $b$.
Hence the generating functions associated with these two expressions are
\begin{align*}
\sum_{n \geq 0}
\frac{q^{b} - q}{q^{b} - q^{a}}
\cdot
q^{n/a - (b-a) \cdot \{b^{-1} \cdot n / a\}}
\cdot z^{n}
& =
\frac{A(z)}{1- q \cdot z^{a}} , \\
\sum_{n \geq 0}
\frac{q^{a} - q}{q^{b}-q^{a}} 
\cdot
q^{n/b + (b-a) \cdot \{a^{-1} \cdot n / b\}} 
\cdot z^{n}
& =
\frac{B(z)}{1- q \cdot z^{b}} ,
\end{align*}
where $A(z)$ and $B(z)$ are polynomials such that
$\deg(A(z)) \leq a-1$ and $\deg(B(z)) \leq b-1$.
Hence the partial fraction decomposition
of the generating function in 
equation~\eqref{equation_generating_function_a_b}
is given by
\begin{align*}
\frac{1- q \cdot z^{ab}}{(1- q \cdot z^{a}) \cdot (1- q \cdot z^{b})}
& =
P(z)
+
\frac{A(z)}{1- q \cdot z^{a}}
-
\frac{B(z)}{1- q \cdot z^{b}} ,
\end{align*}
where $P(z)$ is a polynomial of degree $ab-a-b$.
Hence for $n$ greater than $\deg(P(z)) = ab-a-b$
there is no interference from the polynomial~$P(z)$.
\end{proof}

\begin{example}
{\rm
Let the monoid $M$ be the set $\Nnn - \{1\}$, that is, we are interested in polynomials
where no irreducible factor has multiplicity $1$. 
These polynomials have been studied before;
see~\cite{Stanley,Stong}.
This case is a special
case of the previous proposition, with $a = 2$ and $b = 3$.
Here we have $a^{-1} \equiv 2 \bmod 3$ and $b^{-1} \equiv 1 \bmod 2$.
Note that
$n/2 - \{n/2\} = \lfloor n/2 \rfloor$
and
$n/3 + \{2n/3\} = \lfloor (n-1)/3 \rfloor + 1$.
Equation~\eqref{equation_a_b_improved} simplifies to
\begin{align}
\alpha_{q}(n)
& =
q^{\lfloor n/2 \rfloor}
+
q^{\lfloor n/2 \rfloor - 1}
-
q^{\lfloor (n-1)/3 \rfloor} ,
\label{equation_a=2_b=3}
\end{align}
for $n > 1$.
This is the result of Stanley and Stong~\cite{Stanley,Stong}.
Also note that we obtain that
$\alpha_{q}(6k+2)
=
\alpha_{q}(6k+3)
=
q^{-1} \cdot \alpha_{q}(6k+4)
=
q^{-1} \cdot \alpha_{q}(6k+5)
=
q^{-2} \cdot \alpha_{q}(6k+7)$
for a non-negative integer $k$.
}
\label{example_no_multiplicity_1}
\end{example}

\section{Two results on the number of irreducible factors}
\label{section_Two_results}

Let $f_{j}(p)$ denote the number of irreducible factors of degree~$j$
in the factorization of the polynomial~$p$
including multiplicities.
That is, for $p = x^{6}+x^{4}+x^{3}+x \in F_{2}[x]$
we have $f_{1}(p) = 4$ and $f_{2}(p) = 1$
since $p$ factors as
$x \cdot (x+1)^{3} \cdot (x^{2} + x + 1)$.

Assume that the weights $(w_{{\bf n}})_{{\bf n} \geq {\bf 0}}$ are non-negative real numbers
and that $\alpha_{q}({\bf n})$ in Theorem~\ref{theorem_main}
is non-zero.
Then we can view the ratio
$\wt({\bf p})/\alpha_{q}({\bf n})$
as a probability distribution
on the set
$\{{\bf p} \in F_{q}^{\text{monic}}[x]^{d} : \deg({\bf p}) = {\bf n}\}$.

\begin{proposition}
Assume that $\alpha_{q}({\bf n})$ is non-zero
and pick a polynomial vector ${\bf p} = (p_{1}, \ldots, p_{d})$
at random from 
$\{{\bf p} \in F_{q}^{\text{monic}}[x]^{d} : \deg({\bf p}) = {\bf n}\}$
with the above described distribution.
Then the expected value of $f_{j}(p_{1})$
is given by
\begin{align*}
\frac{M(q;j)}{\alpha_{q}({\bf n})}
\cdot
[{\bf z}^{{\bf n}}]
\frac{\sum_{{\bf n} \geq {\bf 0}} n_{1} \cdot w_{{\bf n}} \cdot {\bf z}^{j \cdot {\bf n}}}
{\sum_{{\bf n} \geq {\bf 0}} w_{{\bf n}} \cdot {\bf z}^{j \cdot {\bf n}}}
\cdot
\prod_{{\bf i} > {\bf 0}}
\left(\frac{1}{1 - q \cdot {\bf z}^{{\bf i}}}\right)^{a_{\bf i}} ,
\end{align*}
where $[{\bf z}^{{\bf n}}]$ denotes the operation
extracting the coefficient of ${\bf z}^{{\bf n}}$.
\label{proposition_expected_number_linear_factors}
\end{proposition}
\begin{proof}
Consider the generating function
\begin{align*}
\sum_{{\bf n} \geq {\bf 0}}
\sum_{\substack{{\bf p} \in F_{q}^{\text{monic}}[x]^{d} \\ \deg({\bf p}) = {\bf n}}}
t^{f_{j}(p_{1})} \cdot \wt({\bf p}) \cdot {\bf z}^{{\bf n}}
& =
\prod_{\substack{s \\ \deg(s) = j}}
\sum_{{\bf n} \geq {\bf 0}} w_{{\bf n}} \cdot t^{n_{1}} \cdot {\bf z}^{j \cdot {\bf n}}
\cdot
\prod_{\substack{s \\ \deg(s) \neq j}}
\sum_{{\bf n} \geq {\bf 0}} w_{{\bf n}} \cdot {\bf z}^{\deg(s) \cdot {\bf n}} \\
& =
\left(
\sum_{{\bf n} \geq {\bf 0}} w_{{\bf n}} \cdot t^{n_{1}} \cdot {\bf z}^{j \cdot {\bf n}}
\right)^{M(q;j)}
\cdot 
\prod_{\substack{k \geq 1 \\ k \neq j}}
\left(
\sum_{{\bf n} \geq {\bf 0}} w_{{\bf n}} \cdot {\bf z}^{k \cdot {\bf n}}
\right)^{M(q,k)} \\
& =
\left(
\frac{\sum_{{\bf n} \geq {\bf 0}} w_{{\bf n}} \cdot t^{n_{1}} \cdot {\bf z}^{j \cdot {\bf n}}}
{\sum_{{\bf n} \geq {\bf 0}} w_{{\bf n}} \cdot {\bf z}^{j \cdot {\bf n}}}
\right)^{M(q,j)}
\cdot 
\prod_{k \geq 1}
\left(
\sum_{{\bf n} \geq {\bf 0}} w_{{\bf n}} \cdot {\bf z}^{k \cdot {\bf n}}
\right)^{M(q,k)} \\
& =
\left(
\frac{\sum_{{\bf n} \geq {\bf 0}} w_{{\bf n}} \cdot t^{n_{1}} \cdot {\bf z}^{j \cdot {\bf n}}}
{\sum_{{\bf n} \geq {\bf 0}} w_{{\bf n}} \cdot {\bf z}^{j \cdot {\bf n}}}
\right)^{M(q,j)}
\cdot 
\prod_{{\bf i} > {\bf 0}}
\left(\frac{1}{1 - q \cdot {\bf z}^{{\bf i}}}\right)^{a_{\bf i}} .
\end{align*}
Apply the partial derivative $\partial/\partial t$
\begin{align*}
&
\sum_{{\bf n} \geq {\bf 0}}
\sum_{\substack{{\bf p} \in F_{q}^{\text{monic}}[x]^{d} \\ \deg({\bf p}) = {\bf n}}}
f_{j}(p_{1}) \cdot t^{f_{j}(p_{1})-1} \cdot \wt({\bf p}) \cdot {\bf z}^{{\bf n}} \\
& =
M(q;j)
\cdot
\frac{\sum_{{\bf n} \geq {\bf 0}} n_{1} \cdot w_{{\bf n}} \cdot t^{n_{1}-1} \cdot {\bf z}^{j \cdot {\bf n}}}
{\sum_{{\bf n} \geq {\bf 0}} w_{{\bf n}} \cdot {\bf z}^{j \cdot {\bf n}}}
\cdot
\left(
\frac{\sum_{{\bf n} \geq {\bf 0}} w_{{\bf n}} \cdot t^{n_{1}} \cdot {\bf z}^{j \cdot {\bf n}}}
{\sum_{{\bf n} \geq {\bf 0}} w_{{\bf n}} \cdot {\bf z}^{j \cdot {\bf n}}}
\right)^{M(q;j)-1}
\cdot 
\prod_{{\bf i} > {\bf 0}}
\left(\frac{1}{1 - q \cdot {\bf z}^{{\bf i}}}\right)^{a_{\bf i}} 
\end{align*}
and
set $t$ to be $1$
\begin{align*}
\sum_{{\bf n} \geq {\bf 0}}
\sum_{\substack{{\bf p} \in F_{q}^{\text{monic}}[x]^{d} \\ \deg({\bf p}) = {\bf n}}}
f_{j}(p_{1}) \cdot \wt({\bf p}) \cdot {\bf z}^{{\bf n}}
& =
M(q;j)
\cdot
\frac{\sum_{{\bf n} \geq {\bf 0}} n_{1} \cdot w_{{\bf n}} \cdot {\bf z}^{j \cdot {\bf n}}}
{\sum_{{\bf n} \geq {\bf 0}} w_{{\bf n}} \cdot {\bf z}^{j \cdot {\bf n}}}
\cdot
\prod_{{\bf i} > {\bf 0}}
\left(\frac{1}{1 - q \cdot {\bf z}^{{\bf i}}}\right)^{a_{\bf i}} .
\end{align*}
The result follows by reading off the coefficient of ${\bf z}^{\bf n}$
and dividing by $\alpha_{q}({\bf n})$.
\end{proof}

Let $f(p)$ denote the number of not-necessarily-distinct irreducible factors in the polynomial $p$,
that is, $f(p) = \sum_{j \geq 1} f_{j}(p)$.
\begin{proposition}
The polynomial in the two variables $q$ and $t$ given by the sum
\begin{align*}
\sum_{\substack{p \in F_{q}^{\text{monic}}[x] \\ \deg(p) = n}} t^{f(p)}
\end{align*}
is symmetric in $q$ and $t$.
\label{proposition_q_t}
\end{proposition}
\begin{proof}
The generating function for the polynomial in the statement is given by
\begin{align*}
\sum_{p \in F_{q}^{\text{monic}}[x]} t^{f(p)} \cdot z^{\deg(p)}
& =
\prod_{k \geq 1}
\left(
\frac{1}{1 - t \cdot z^{k}}
\right)^{M(q,k)} \\
& =
\prod_{k \geq 1}
\left(
\prod_{j \geq 1}
\left(
\frac{1}{1 - z^{k \cdot j}}
\right)^{M(t,j)}
\right)^{M(q,k)} .
\end{align*}
Observe that the last expression is symmetric in $q$ and $t$,
proving the statement.
\end{proof}

\begin{example}
{\rm
For degrees $n = 2, 3, 4$ we obtain the following three polynomials
in Proposition~\ref{proposition_q_t}:
\begin{align*}
&
\frac{1}{2!}
\cdot
\begin{pmatrix}
t^{2} \\ t
\end{pmatrix}^{T}
\cdot
\begin{pmatrix}
1 & 1 \\
1 & -1
\end{pmatrix}
\cdot
\begin{pmatrix}
q^{2} \\ q
\end{pmatrix} , \:\:\:\:\:\:
\frac{1}{3!}
\cdot
\begin{pmatrix}
t^{3} \\ t^{2} \\ t
\end{pmatrix}^{T}
\cdot
\begin{pmatrix}
1 & 3 & 2 \\
3 & -3 & 0 \\
2 & 0 & -2
\end{pmatrix}
\cdot
\begin{pmatrix}
q^{3} \\ q^{2} \\ q
\end{pmatrix} , \\
&
\frac{1}{4!}
\cdot
\begin{pmatrix}
t^{4} \\ t^{3} \\ t^{2} \\ t
\end{pmatrix}^{T}
\cdot
\begin{pmatrix}
1 & 6 & 11 & 6 \\
6 & 0 & -6 & 0 \\
11 & -6 & 1& -6 \\
6 & 0 & -6 & 0
\end{pmatrix}
\cdot
\begin{pmatrix}
q^{4} \\ q^{3} \\ q^{2} \\ q
\end{pmatrix} .
\end{align*}
}
\end{example}

\section{Open questions}

For any of the results in equations~\eqref{equation_M_1},
\eqref{equation_M_r=1},
\eqref{equation_M_d=1},
\eqref{equation_d_r_m},
\eqref{equation_a_b},
\eqref{equation_a_b_improved}
and~\eqref{equation_a=2_b=3},
are there bijective proofs?
Furthermore, do these enumerations have a connection with
algebraic geometry?
In the papers~\cite{Church_Ellenberg_Farb,Fulman}
the expected number of linear factors in a square-free polynomial
is determined. Are there other explicit results that follow from
Proposition~\ref{proposition_expected_number_linear_factors}?
Is there an application of Theorem~\ref{theorem_main}
where the product expression in~\eqref{equation_product}
has an infinite number of factors that differ from~$1$?
In all our examples, this product has been finite.

\section*{Acknowledgements}

The author thanks David Leep for the references
\cite{Alegre_Juarez_Pajela,Fulman}.
The author also thanks the referee and
Theodore Ehrenborg for their comments on an earlier draft.
The author thanks Cornell University for support
during a research visit in October 2023, where parts
of this paper were written.
This work was partially supported by a grant from the
Simons Foundation
(\#854548 to Richard Ehrenborg).

\newcommand{\journal}[6]{{\sc #1,} #2, {\it #3} {\bf #4} (#5), #6.}
\newcommand{\book}[4]{{\sc #1,} ``#2,'' #3, #4.}
\newcommand{\bookf}[5]{{\sc #1,} ``#2,'' #3, #4, #5.}
\newcommand{\arxiv}[3]{{\sc #1,} #2, {\tt #3}.}
\newcommand{\preprint}[3]{{\sc #1,} #2, preprint {(#3)}.}
\newcommand{\preparation}[2]{{\sc #1,} #2, in preparation.}
\newcommand{\toappear}[3]{{\sc #1,} #2, to appear in {\it #3}.}


\begin{thebibliography}{99}

\bibitem{Alegre_Juarez_Pajela}
\preprint{M.\ Alegre, P.\ Juarez and H.\ Pajela}
             {Statistics about polynomials over finite fields}
             {2015}

\bibitem{Athanasiadis}
\journal{C.\ A.\ Athanasiadis}
            {Characteristic polynomials of subspace arrangements and finite fields}
            {Adv.\ Math.}
            {122}{1996}{193--233}

\bibitem{Beck_Robins}
\book{M.\ Beck and S.\ Robins}
         {Computing the continuous discretely.
          Integer-point enumeration in polyhedra}
         {Second edition, Undergrad.\ Texts Math.\ Springer, New York}{2015}


\bibitem{Church_Ellenberg_Farb}
{\sc T.\ Church, J.\ S.\ Ellenberg and B.\ Farb,}
{Representation stability in cohomology and asymptotics for
              families of varieties over finite fields},
              in
            {Algebraic topology: applications and new directions},
            {\it Contemp.\ Math.},
            {\bf 620} (2014), {1--54},
            {Amer. Math. Soc., Providence, RI},

\bibitem{Fulman}
\journal{J.\ Fulman}
            {A generating function approach to counting theorems for
             square-free polynomials and maximal tori}
             {Ann. Comb.}
             {10}{2016}{587--599}

\bibitem{Lothaire}
{\sc M.\ Lothaire},
    Combinatorics on words.
Encyclopedia Math.\ Appl., 17
Addison-Wesley Publishing Co., Reading, MA, 1983.

\bibitem{Metropolis_Rota_1}
\journal{N.\ Metropolis and G.-C.\ Rota}
            {Witt vectors and the algebra of necklaces}
            {Adv.\ in Math.}
            {50}{1983}{95--125}

\bibitem{Metropolis_Rota_2}
{\sc N.\ Metropolis and G.-C.\ Rota,}
The cyclotomic identity,
{\it Contemp.\ Math.}, {\bf 34}
American Mathematical Society, Providence, RI, 1984, 19--27.


\bibitem{Reutenauer}
\journal{C.\ Reutenauer}
            {Mots circulaires et polyn\^omes irr\'eductibles}
             {Ann.\ Sci.\ Math.\ Qu\'ebec}
             {12}{1988}{275--285}   

\bibitem{EC1}
\book{R.\ P.\ Stanley}
         {Enumerative Combinatorics, Vol. I}
         {Cambridge University Press}
         {2012}

         
\bibitem{Stanley}
\journal{R.\ P.\ Stanley}
            {Problem 11348}
            {Amer.\ Math.\ Monthly}
            {115}{2008}{page 262}

\bibitem{Stanley_Some}
\preprint{R.\ P.\ Stanley}
             {Some enumerative applications of cyclotomic polynomials}
             {2024}

\bibitem{Stong}
\journal{R.\ Stong}
            {Solution to Problem 11348}
            {Amer.\ Math.\ Monthly}
            {117}{2010}{87--88}
\end{thebibliography}
\end{document}